\documentclass{amsart}
\usepackage{eurosym}
\usepackage{amssymb}
\usepackage{amsmath}
\usepackage{amsfonts}
\usepackage{graphicx,color}

\newtheorem{theorem}{Theorem}
\theoremstyle{plain}

\newtheorem{definition}{Definition}

\newtheorem{lemma}{Lemma}

\newtheorem{proposition}{Proposition}
\newtheorem{remark}{Remark}

\numberwithin{equation}{section}

\begin{document}
\title[A sub-super solution method for a class of nonlocal problems... ]{A sub-super solution method for a class of nonlocal problems involving the $p(x)-$Laplacian operator and applications}

\begin{abstract}
In the present paper we study the existence of solutions for some nonlocal problems involving the $p(x)-$Laplacian operator. The approach is based on a new sub-supersolution method. 
\end{abstract}

\author{Gelson C.G. dos Santos}
\address{Gelson C.G. dos Santos\\
Universidade Federal do Par\'{a}, Faculdade de Matem\'{a}tica, CEP: 66075-110 Bel\'{e}m-PA, Brazil.}
\email{cgelson@ymail.com}
\author{Giovany M. Figueiredo}
\address{Giovany M. Figueiredo \\
Universidade de Brasília, Departamento de Matem\'{a}tica , CEP: 70910-900 Brasília-DF, Brazil .}
\email{giovany@unb.br}
\author{Leandro da S. Tavares}
\address{Leandro da S. Tavares\\
Universidade Federal do Cariri, Centro de Ci\^{e}ncias e Tecnologia, CEP:63048-080, Juazeiro do Norte-CE,
Brazil.}
\email{leandro.tavares@ufca.edu.br}
\maketitle

\section{Introduction}
In this work we are interested in the nonlocal problem
$$
\left\{\begin{array}{rcl}\label{problema-(P)}
-\mathcal{A}(x,|u|_{L^{r(x)}})\Delta_{p(x)} u&=&f(x,u)|u|_{L^{q(x)}}^{\alpha(x)}+g(x,u)|u|_{L^{s(x)}}^{\gamma(x)} \;\;\mbox{in} \;\;\Omega,\\
\vspace{.2cm}
u&=&0\;\;\mbox{on}\;\;\partial\Omega,
\end{array}
\right. \leqno{(P)}
$$
]where $\Omega$ is a bounded domain in $\mathbb{R}^{N} (N > 1)$ with  $C^{2}$  boundary, $|.|_{L^m(x)}$ is the norm of the space $L^{m(x)}(\Omega),$ $-\Delta_{p(x)}u:=-div(|\nabla u|^{p(x)-2}\nabla u)$ is the  $p(x)-$Laplacian operator, $r,q, s,\alpha,\gamma:\Omega\rightarrow[0,\infty)$ are measurable functions and $\mathcal{A},f,g:\overline{\Omega}\times\mathbb{R}\rightarrow\mathbb{R}$ are continuous functions satisfying certain conditions.

In the last decades several works related to the  $p(x)-$Laplacian operator arose, see for instance \cite{acerbi-mingione, fan-zao, fan-regular, fan-zang, fan, Fan, Liu} and the references therein. Partial differential equations involving the $p(x)-$Laplacian arise, for instance, in nonlinear elasticity, fluid mechanics, non-Newtonian fluids and image processing. See for instance \cite{acerbi-mingione, chen-levine, ruzicka} and the references therein for more informations.

The nonlocal term $|.|_{L^{m(x)}}$ with the condition $p(x)=r(x)\equiv2$  is considered in the well know Carrier's equation
$$\rho u_{tt}-a(x,t,|u|_{L^{2}}^{2})\Delta u=0$$
that models the vibrations of a elastic string  when the variation of the tensions are not too small. See \cite{carrier} for more details. The same nonlocal term arises also in Population Dynamics, see \cite{chipot-chang1, chipot-lovat} and its references.

In the litarature there are several works related to  $(P)$ but with  $p(x)\equiv p$ ($p$  constant), see for instance \cite{alves-covei, chipot-chang1, chipot-chang2,chipot-chang3, chipot-correa, chipot-lovat, Chipot-Lovat, chipot-molinet, chipot-rodrigues, chipot-savitska1, chipot-savitska2, chipot-valente-caffarelli, correa-figueiredo, correa-menezes, correa-antonio, gel-giovany, yan-wang}. For example in  \cite{correa-figueiredo} the authors used a sub-supersolution argument  to study the nonlocal problem
$$\left\{\begin{array}{rcl}
-\Delta_{p} u=|u|_{L^{q(x)}}^{\alpha(x)}\;\;\mbox{em}\;\;\Omega,\\
\vspace{.2cm}
u=0\;\;\mbox{sobre}\;\;\partial\Omega.
\end{array}
\right.$$

In \cite{alves-covei}, the authors used an abstract  sub-super solution Theorem whose proof is chiefly based on a version of the Minty-Browder  Theorem for pseudomonotone operators to study the problem
$$
\left\{\begin{array}{rcl}
-a(\int_{\Omega}|u|^{q})\Delta u&=&h_{1}(x,u)f(\int_{\Omega}|u|^{p})+h_{2}(x,u)g(\int_{\Omega}|u|^{r})\;\;\mbox{in}\;\;\Omega,\\
\vspace{.2cm}
u&=&0\;\;\mbox{on}\;\;\partial\Omega,
\end{array}
\right. 
$$
where $h_{i}:\overline{\Omega}\times\mathbb{R}^{+}\rightarrow\mathbb{R}$ are continuous functions, $q,p,r\in[1,\infty)$ are constants and the functions $a,f,g:[0,\infty)\rightarrow\mathbb{R}^{+}$ with  $ f,g \in L^{\infty}([0,\infty))$ and
$$\begin{array}{rcl}
a(t), f(t), g(t)\geq a_{0}>0
\end{array}
$$
for all $t \in[0,\infty)$ where $a_0$ is a constant.

Recently in \cite{yan-wang}, the authors studied the existence and multiplicity of solutions to the problem
$$\left\{\begin{array}{rcl}\label{Y.W}
-a(\int_{\Omega}|u|^{\gamma})\Delta u=f_{\lambda}(x,u)\;\;\mbox{em}\;\;\Omega,\\
\vspace{.2cm}
u=0\;\;\mbox{sobre}\;\;\partial\Omega,
\end{array}
\right.
$$
where $\gamma\in(0,\infty),$ $a(t)\geq a_{0}>0$ where $a_0$ is a constant and $f_{\lambda}:\Omega\times\mathbb{R}\rightarrow\mathbb{R}$ are continuous functions with $f_{\lambda}$ depending on the parameter $\lambda.$

In \cite{chipot-correa}, the authors used the Schauder's Fixed Point Theorem to study the boundary value problem
$$\left\{\begin{array}{rcl}
-\mathfrak{A}(x,u)\Delta u=\lambda f(u)\;\;\mbox{em}\;\;\Omega,\\
\vspace{.2cm}
u=0\;\;\mbox{sobre}\;\;\partial\Omega,
\end{array}
\right. 
$$
where $f\in C^{1}([0,\theta],\mathbb{R}),$ $f(0)=0=f(\theta),$  $f'(0)>0,$ $f(t)>0$ in $(0,\theta),$ the function $\mathfrak{A}:\Omega\times L^{p}(\Omega)\rightarrow\mathbb{R}$ is such that the mapping $x\mapsto\mathfrak{A}(x,u)$ is measurable for all $u \in L^{p}(\Omega)$ and the function $u\mapsto\mathfrak{A}(x,u)$ is continuous from $L^{p}(\Omega)$ into $\mathbb{R}$ for almost all $x \in \Omega.$ They also considered that  there are constants $a_{0},a_{\infty}>0$  such that
$$\begin{array}{rcl}
a_{0}\leq\mathfrak{A}(x,u)\leq a_{\infty}\;\mbox{a.e}\; \mbox{in} \; \Omega
\end{array}
$$
for all $u\in L^{p}(\Omega).$

Recently in \cite[Theorem 1]{gel-giovany} the first two authors considered the problem $(P)$ for  $p(x)\equiv2$ (i.e, $-\Delta_{p(x)}=-\Delta$). They proved a sub-super solution theorem for $(P)$ and applied such result in three problems. Specifically they considered a sublinear problem, a concave-convex problem and a logistic equation. Their arguments are chiefly based on the existence of the first eigenvalue of the Laplacian operator $(-\Delta, H^{1}_{0}(\Omega))$. The $p(x)-$Laplacian operator, in general,  has no first eigenvalue, that is, the infimum of the eigenvalues  equals 0 (see \cite{fan-eigen}).

The lack of the existence of the first eigenvalue implies a considerable difficult to consider boundary values problems  involving the $p(x)-$Laplacian by using sub-supersolution methods. Papers that consider such problems by using the mentioned method  are rare in the literature. Among  such works we want to mention the papers \cite{alves-moussaoui-tavares, alves-moussaoui,Liu, ying-yang, zhang}.

The main goal of this paper is to generalize \cite[Theorem 1]{gel-giovany} for the  $p(x)-$Laplacian operator and the three applications of  \cite{gel-giovany}. Below we describe the main points regarding the generalization of the results in \cite{gel-giovany}.

\begin{itemize}
	
	\item[(i)] In \cite{gel-giovany} the homogeneity of $(-\Delta,H^{1}_{0}(\Omega)) $ and the  eigenfunction associated to the first eigenvalue  $\lambda_{1}$ are used to construct a subsolution. Differently from the  operator $(-\Delta_{p}, W^{1,p}_{0}(\Omega))$, the  $p(x)-$Laplacian operator is not homogeneneous. Another important point is the problem of the existence of the first eigenvalue of the  operator $\big(-\Delta_{p(x)}, W_{0}^{1,p(x)}(\Omega)\big)$. To avoid these problems we used some arguments contained in \cite{Liu};

\item[(ii)] We present weaker conditions on the exponents $r,q, s,\alpha$ and $\gamma;$

\item[(iii)] As an application of Theorem \ref{theorem-to-(P)} we prove the existence of a positive solution  for some nonlocal problems which generalize the three problems considered in \cite[Theorem 1]{gel-giovany};

\item[(iv)] As in \cite[Theorem 1]{gel-giovany} and differently from several works that consider the nonlocal term  $\mathcal{A}(x,|u|_{L^{r(x)}})$ satisfying $\mathcal{A}(x,t)\geq a_{0}>0$ (where $a_0$ is a constant) our  Theorem \ref{theorem-to-(P)} allow us to study  $(P)$ in the mentioned case  and in situations where $\mathcal{A}(x,0)=0.$
\end{itemize}

In this work we will assume that the functions $r,p,q, s,\alpha$ and $\gamma$ satisfy the hipothesis below
\begin{itemize}
	\item[($H_{0}$)]$p\in C^{1}(\overline{\Omega}), r,q,s\in L_{+}^{\infty}(\Omega)$ and $\alpha,\gamma\in L^{\infty}(\Omega)$ satisfy
	$$1<p^{-}:=\inf_{\Omega} p(x)\leq p^{+}:=\sup_{\Omega} p(x)<N\;\;\text{e}\;\;\alpha(x),\gamma(x)\geq0\;\text{a.e in}\;\Omega.$$
\end{itemize}

Now we present our main result:

\begin{theorem}\label{theorem-to-(P)}
	Suppose that  $r,p,q,s,\alpha$ and $\gamma$ satisfy  $(H_{0}),$ $(\underline{u},\overline{u})$ is a pair of sub-super solution for  $(P)$ with $\underline{u}>0\;\text{a.e in}\;\Omega,$ $f(x,t),g(x,t)\geq0$ in $\overline{\Omega}\times[0, |\overline{u}|_{L^{\infty}}]$ are continuous functions and $\mathcal{A}:\overline{\Omega}\times(0,\infty)\rightarrow\mathbb{R}$ is continuos with
	$\mathcal{A}(x,t)>0\;\text{in}\;\overline{\Omega}\times\bigl[|\underline{u}|_{L^{r(x)}},\;|\overline{u}|_{L^{r(x)}}\bigl].$
	Then $(P)$ has a weak positive solution $u\in[\underline{u},\overline{u}].$
\end{theorem}

\section{Preliminaries: The spaces $L^{p(x)}(\Omega),$ $W^{1,p(x)}(\Omega)$ and $W_{0}^{1,p(x)}(\Omega)$}

In this section we will point some facts regarding to the spaces $L^{p(x)}(\Omega),$  $W^{1,p(x)}(\Omega)$ and $W_{0}^{1,p(x)}(\Omega)$ that will be often used in this work.  For more information see Fan-Zhang \cite{fan} and the references therein.

Let $\Omega\subset I\!\!R^{N} ( N \geq 1)$ be a bounded domain. We define the space
$$L_{+}^{\infty}(\Omega)=\left\{ m\in L^{\infty}(\Omega)\;\text{with}\; ess \inf m(x)\geq1\right\}.$$

Given $p\in L_{+}^{\infty}(\Omega)$ we define the generalized  Lebesgue space

$$
L^{p(x)}(\Omega )=\left\{u\in\mathcal{S}(\Omega): \int_{\Omega}|u(x)|^{p(x)}dx<\infty \right\},
$$
where $\mathcal{S}(\Omega):=\biggl\{u:\Omega\rightarrow\mathbb{R}: u\; \text{is measurable}\biggl\}$.\\

We define in $L^{p(x)}(\Omega)$ the norm

$$
|u|_{p(x)}:=\inf \left\{\lambda>0;\; \int_{\Omega}\left|\frac{u(x)}{\lambda}\right|^{p(x)}dx\leq 1\right\}.
$$
The space $(L^{p(x)}(\Omega), |.|_{L^{p(x)}})$ is a Banach space.

Given $m\in L^{\infty}(\Omega)$ we define
$$
m^{+}:=ess \sup_{\Omega}m(x)\;\;\text{e}\;\;m^{-}:=ess \inf_{\Omega}m(x).
$$

\begin{proposition}\label{norm-property} Define the quantity $\rho(u): =\int_{\Omega}|u|^{p(x)}dx.$ For all $u, u_{n}\in L^{p(x)}(\Omega), n \in \mathbb{N}$  the following assertions hold
	
	\noindent {\bf (i)} Let $u\neq 0$ in $L^{p(x)}(\Omega)$. Then $|u|_{L^{p(x)}}=\lambda\Leftrightarrow \rho(\frac{u}{\lambda})=1,$
	
	\noindent {\bf (ii)} If $|u|_{L^{p(x)}}< 1 \;(= 1;\;> 1)$  then $\rho(u) < 1\; (= 1;\;> 1),$
	
	\noindent {\bf (iii)} If $|u|_{L^{p(x)}}> 1 $ then $|u|_{L^{p(x)}}^{p^{-}}\leq\rho(u)\leq|u|_{L^{p(x)}}^{p^{+}},$
	
	\noindent {\bf (iv)} If $|u|_{L^{p(x)}}< 1 $ then $|u|_{L^{p(x)}}^{p^{+}}\leq\rho(u)\leq|u|_{L^{p(x)}}^{p^{-}},$
	
	\noindent {\bf (v)} $|u_{n}|_{L^{p(x)}}\rightarrow0\Leftrightarrow\rho(u_{n})\rightarrow0\;$ and $\;|u_{n}|_{L^{p(x)}}\rightarrow\infty\Leftrightarrow\rho(u_{n})\rightarrow\infty.$
\end{proposition}

\begin{theorem}
	Let $p,q\in L_{+}^{\infty}(\Omega)$. The following statements hold
	
	\noindent {\bf (i)} If $p^{-}>1$ and $\frac{1}{q(x)}+\frac{1}{p(x)}=1$ a.e in $\Omega$ then
	$
	\left|\int_{\Omega}uvdx\right|\leq \bigl(\frac{1}{p^{-}}+\frac{1}{q^{-}}\bigl)|u|_{L^{p(x)}}|v|_{L^{q(x)}}.
	$
	
	\noindent {\bf (ii)} If $q(x)\leq p(x),\;\text{a.e in}\;\Omega$ and $|\Omega|<\infty$  then $L^{p(x)}(\Omega)\hookrightarrow L^{q(x)}(\Omega).$
\end{theorem}

We define the generalized Sobolev space as 
$$W^{1,p(x)}(\Omega):=\left\{u\in L^{p(x)}(\Omega):\frac{\partial u}{\partial x_{j}}\in L^{p(x)}(\Omega),j=1,...,N\right\}$$
with the norm $\|u\|_{*}=|u|_{L^{p(x)}}+\sum_{j=1}^{N}\big|\frac{\partial u}{\partial x_{j}}\big|_{L^{p(x)}}, u\in W^{1,p(x)}(\Omega).$ The space $W_{0}^{1,p(x)}(\Omega)$ is defined as the closure of  $C_{0}^{\infty}(\Omega)$ with respect to the norm $\|.\|_{*}.$

\begin{theorem} If $p^{-}>1$ then $W^{1,p(x)}(\Omega)$ is a  Banach, separable and reflexive space.
	
\end{theorem}

\begin{proposition}\label{embeddings} Let $\Omega\subset\mathbb{R}^{N}$ be a bounded domain and consider $p,q\in C(\overline{\Omega}).$ Define the function $p^{*}(x)=\frac{Np(x)}{N-p(x)}$ if $p(x)<N$ and $p^{*}(x)=\infty$ if $N\geq p(x).$ The following statements hold
	
	\noindent {\bf (i)} (Poincar\'{e} inequality) If $p^{-}>1$ then there is a constant $C > 0$ such that $|u|_{L^{p(x)}}\leq C|\nabla u|_{L^{p(x)}}$ for all $u\in W_{0}^{1,p(x)}(\Omega).$
\\
	
	\noindent {\bf (ii)} If $p^{-},q^{-} > 1$ and $q(x)<p^{*}(x)$ for all $x\in\overline{\Omega},$ the embedding
	$W^{1,p(x)}(\Omega) \hookrightarrow L^{q(x)}(\Omega)$ is continuous and compact.
\end{proposition}

From $(i)$ of Proposition \ref{embeddings} we have that  $|\nabla u|_{L^{p(x)}}$ defines a norm in  $W_{0}^{1,p(x)}(\Omega)$ which is equivalent to the norm $\|.\|_{*}.$

\begin{definition}Consider $u,v\in W^{1,p(x)}(\Omega)$. We say that $-\Delta_{p(x)}u\leq -\Delta_{p(x)}v,$ if
$$\int_{\Omega}|\nabla u|^{p(x)-2}\nabla u \nabla\varphi\leq\int_{\Omega}|\nabla v|^{p(x)-2}\nabla v \nabla\varphi,$$	
 for all $\varphi\in W_{0}^{1,p(x)}(\Omega)$ with $\varphi\geq0.$
\end{definition}
The following result is contained in \cite[Lemma 2.2]{fan-zang} and \cite[Proposition 2.3]{Fan}.
\begin{proposition}\label{PC} Consider $u,v\in W^{1,p(x)}(\Omega).$ If $-\Delta_{p(x)}u\leq-\Delta_{p(x)}v$ and $u\leq v$ on $\partial\Omega,$ (i.e., $(u-v)^{+}\in W_{0}^{1,p(x)}(\Omega)$) then $u\leq v$ in $\Omega.$ If $u,v\in C(\overline{\Omega})$ and the set $S=\big\{x\in\Omega: u(x)=v(x)\big\}$ is a compact set of $\Omega$ then $S=\emptyset.$ 
\end{proposition}

\begin{lemma}\cite[Lemma 2.1]{Fan}\label{Fan}
	Let $\lambda>0$ be the unique solution of the problem
	\begin{equation}\label{probl-linear-lambda}
	\begin{aligned}
	\left\{\begin{array}{rcl}
	-\Delta_{p(x)}z_{\lambda} &=&\lambda\;\;\mbox{in} \;\;\Omega,\\
	\vspace{.2cm}
	u&=&0\;\;\mbox{on}\;\;\partial\Omega.
	\end{array}
	\right.
	\end{aligned}
	\end{equation}
Define $\rho_{0}=\frac{p^{-}}{2|\Omega|^{\frac{1}{N}}C_{0}}.$  If $\lambda\geq \rho_{0}$ then $|z_{\lambda}|_{L^{\infty}}\leq C^{*}M^{\frac{1}{p^{-}-1}}$ and $|z_{\lambda}|_{L^{\infty}}\leq C_{*}M^{\frac{1}{p^{+}-1}}$ if  $\lambda<\rho_{0}.$
Here  $C^{*}$ and $C_{*}$ are positive constants dependending only on $p^{+},p^{-},N,|\Omega|$ and $C_{0}$ where  $C_{0}$ is the best constant of the embedding  $W_{0}^{1,1}(\Omega) \hookrightarrow L^{\frac{N}{N-1}}(\Omega)$.
\end{lemma}
Regarding to the function  $z_{\lambda}$ of the previous result, it follows from  \cite[Theorem 1.2]{fan-regular} and  \cite[Theorem 1]{fan-zang} that   $z_{\lambda} \in C^{1}(\overline{\Omega}) $ with $z_\lambda >0$ in $\Omega.$

\section{Proof of Theorem \ref{theorem-to-(P)}} 

The goal of this section is the proof of Theorem \ref{theorem-to-(P)}. First we need some definitions.

We say that  $u\in W_{0}^{1,p(x)}(\Omega)\bigcap L^{\infty}(\Omega)$ is a  (weak) solution of $(P)$ if
$$\int_{\Omega}|\nabla u|^{p(x)-2}\nabla u\nabla\varphi=\int_{\Omega}\left(\frac{f(x,u)|u|_{L^{q(x)}}^{\alpha(x)}}{\mathcal{A}(x,|u|_{L^{r(x)}})}+\frac{g(x,u)|u|_{L^{s(x)}}^{\gamma(x)}}{\mathcal{A}(x,|u|_{L^{r(x)}})}\right)\varphi$$
for all $\varphi\in W_{0}^{1,p(x)}(\Omega).$

Given $u,v\in \mathcal{S}(\Omega)$ we write $u\leq v$ if $u(x)\leq v(x)$ a.e in $\Omega$ and  $[u,v]:=\bigl\{w\in \mathcal{S}(\Omega): u(x)\leq w(x)\leq v(x)\;\text{a.e in}\;\Omega\bigl\}.$

We say that  $(\underline{u},\overline{u})$ is a sub-super solution pair for $(P)$  if $\underline{u}\in$ $W_{0}^{1,p(x)}(\Omega)\bigcap L ^{\infty}(\Omega),$ $\overline{u}\in W^{1,p(x)}(\Omega)\bigcap L ^{\infty}(\Omega)$ and satisfy the inequalities  $\underline{u}\leq\overline{u}$, $\underline{u}=0\leq\overline{u}\;\;\text{ on}\;\;\partial\Omega$ and if for all $\varphi\in W_{0}^{1,p(x)}(\Omega)$ with $\varphi\geq0$ the following inequalities hold

\begin{equation}\label{eq1.1}
\int_{\Omega}|\nabla \underline{u}|^{p(x)-2}\nabla\underline{u}\nabla\varphi\leq \int_{\Omega}\left(\frac{f(x,\underline{u})|\underline{u}|_{L^{q(x)}}^{\alpha(x)}}{\mathcal{A}(x,|w|_{L^{r(x)}})}+\frac{g(x,\underline{u})|\underline{u}|_{L^{s(x)}}^{\gamma(x)}}{\mathcal{A}(x,|w|_{L^{r(x)}})}\right)\varphi
\end{equation}
and
\begin{equation}\label{eq1.2}
\int_{\Omega}|\nabla \overline{u}|^{p(x)-2}\nabla\overline{u}\nabla\varphi\geq \int_{\Omega}\left(\frac{f(x,\overline{u})|\overline{u}|_{L^{q(x)}}^{\alpha(x)}}{\mathcal{A}(x,|w|_{L^{r(x)}})}+\frac{g(x,\overline{u})|\overline{u}|_{L^{s(x)}}^{\gamma(x)}}{\mathcal{A}(x,|w|_{L^{r(x)}})}\right)\varphi
\end{equation}
for all $w\in [\underline{u},\overline{u}].$

\begin{proof}[Proof of Theorem \ref{theorem-to-(P)}] Consider the operator $T:L^{p(x)}(\Omega)\rightarrow L^{\infty}(\Omega)$ defined by
	$$(Tu)(x)=\left\{\begin{array}{rcl}
	\underline{u}(x)\;\;\;\;\;\;\mbox{if}\;\;\;\;\;\;\;\;\;\;\;\;\;u(x)\leq\underline{u}(x),\\
	\vspace{.2cm}
  	 u(x) \;\;\;\;\; \mbox{if}\;\;\underline{u}(x)\leq u(x)\leq\overline{u}(x),\\
	\vspace{.2cm}
	\overline{u}(x)\;\;\;\;\;\;\mbox{if}\;\;\;\;\;\;\;\;\;\;\;\;\;u(x)\geq\overline{u}(x).
	\end{array}
	\right.$$
The operator $T$ is well defined because $\underline{u},\overline{u}\in L^{\infty}(\Omega)$ and $Tu\in[\underline{u},\overline{u}]$  . Let $p'(x)=\frac{p(x)}{p(x)-1}$ and consider the operator
$$H:[\underline{u},\overline{u}]\rightarrow L^{p'(x)}(\Omega)$$
$$H(v)(x)=\frac{f(x,v(x))|v|_{L^{q(x)}}^{\alpha(x)}}{\mathcal{A}(x,|v|_{L^{r(x)}})}+\frac{g(x,v(x))|v|_{L^{s(x)}}^{\gamma(x)}}{\mathcal{A}(x,|v|_{L^{r(x)}})},$$
where $|.|_{L^{m(x)}}$ denotes the norm of  $L^{m(x)}(\Omega).$

Note that the operators  $H$ and $u\mapsto HoT(u)$ are well defined. In fact, since $f,g$ and $\mathcal{A}$ are continuous functions with $\mathcal{A}(x,t)>0$ in the compact set $\overline{\Omega}\times\bigl[|\underline{u}|_{L^{r(x)}},| \overline{u}|_{L^{r(x)}}\bigl]$ and $|w|_{L^{m(x)}}^{\theta(x)}\leq |w|_{L^{m(x)}}^{\theta^{-}}+|w|_{L^{m(x)}}^{\theta^{+}}$ for all $ w\in L^{m(x)}(\Omega),\theta\in L^{\infty}(\Omega)$ then there is a constant $K_{0}>0$ such that
\begin{equation*}
|H(v)|\leq K_{0},\;\forall\;v\in[\underline{u},\overline{u}].
\end{equation*}
Since  $\Omega$ is a bounded domain it follows that $H$ is well defined. The operator $u\mapsto HoT(u)$ is well defined because the inclusion $Tu \in [\underline{u}, \overline{u}]$ implies that 
\begin{equation}\label{eq1.3}
 |H(Tu)  | \leq K_0 
\end{equation}
for all $u \in L^{p(x)}(\Omega).$

We claim that the operator $u\mapsto HoT(u)$ is continuous. In order to show such affirmation let $(u_n)$ be a sequence in $L^{p(x)}(\Omega)$ that converges to $u$ in $L^{p(x)}(\Omega).$ Since $T u_n, Tu \in [\underline{u},\overline{u}]$ the Lebesgue Dominated Convergence Theorem combined with Proposition \ref{norm-property} implies that $T u_n \rightarrow T u$ in $L^{m(x)}(\Omega)$ for all $ m \in L^{\infty}_{+}(\Omega).$ The continuity of $f,g$ and $\mathcal{A}$ combined with the Lebesgue Dominated Convergence Theorem implies that $H(Tu_n) \rightarrow H(Tu)$ in $L^{p^{\prime}(x)}(\Omega)$ and then we have  the desired continuity.

Fix $v \in L^{p(x)}(\Omega)$. The inequality \eqref{eq1.3} implies that $(H o T)(v) \in L^{\infty}(\Omega)$, thus by \cite[Theorem 4.2]{fan-zhang} the problem
$$\left\{\begin{array}{rcl}
-\Delta_{p(x)} u&=&H(Tv) \;\;\text{in}\;\;\Omega,\\
\vspace{.2cm}
u&=&0\;\;\text{on}\;\;\partial\Omega.
\end{array}
\right.\leqno{(P_{L})}$$
has a unique solution. Therefore we can define an operator $S:L^{p(x)}(\Omega)\rightarrow L^{p(x)}(\Omega),$ given by $S(v)=u$ where $u\in W_{0}^{1,p(x)}(\Omega)$ is the unique solution of $(P_{L})$.

We affirm that $S$ is compact. In fact let $(v_n) $ be a bounded sequence in $L^{p(x)}(\Omega)$ and define $u_n:= S(v_n), n \in \mathbb{N}.$ The definition of $S$ implies that 

$$\int_{\Omega}|\nabla u_{n}|^{p(x)-2}\nabla u_{n}\nabla\varphi=\int_{\Omega}H(Tv_{n})\varphi$$
for all $n \in \mathbb{N}$ and $\varphi\in W_{0}^{1,p(x)}(\Omega).$ Using the inclusion $Tv_{n}\in[\underline{u},\overline{u}],$ the inequality   $(\ref{eq1.3})$ and  considering  the test function $\varphi=u_{n}$  we have
$$\int_{\Omega}|\nabla u_{n}|^{p(x)}\leq K_{0}\int_{\Omega}|u_{n}|$$
for all $n \in \mathbb{N}.$

The Poincar\'{e} inequality combined with the embedding $L^{p(x)}(\Omega) \hookrightarrow L^{1}(\Omega)$ implies that
$$\int_{\Omega}|\nabla u_{n}|^{p(x)}\leq C\|u_{n}\|$$
for all $n \in \mathbb{N}$ where $C$ is a constant that does not depend on $n \in \mathbb{N}.$
	
If  $\|u_{n}\|>1 $ then by Proposition \ref{norm-property} we have
$$\|u_{n}\|^{p^{-}}\leq C\|u_{n}\|$$	
for all $n \in \mathbb{N}$ where the constant $C$ does not depend on $n \in \mathbb{N}.$ Therefore the sequence $(u_n)$ is bounded in $W^{1,p(x)}_{0}(\Omega).$ Thus, up to a subsequence , we have $u_n \rightharpoonup u$ in $W^{1,p(x)}_{0}(\Omega)$ for some $u \in W^{1,p(x)}_{0}(\Omega)$. Since the embedding $ W^{1,p(x)}_{0}(\Omega) \hookrightarrow L^{p(x)}(\Omega)$ is compact we have $u_n \rightarrow u$ in $L^{p(x)}(\Omega)$. Therefore $S$  is a compact operator.

With respect to  continuity, let $(v_n)$ a sequence in $L^{p(x)}(\Omega)$ with $v_n \rightarrow v$ in $L^{p(x)}(\Omega)$ for $v \in L^{p(x)}(\Omega).$ Define $u_n:= S(v_n)$ and $u: = S(v).$ Note that

$$\int_{\Omega}|\nabla u_{n}|^{p(x)-2}\nabla u_{n}\nabla\varphi=\int_{\Omega}H(Tv_{n})\varphi$$
and
$$\int_{\Omega}|\nabla u|^{p(x)-2}\nabla u\nabla\varphi=\int_{\Omega}H(Tv)\varphi$$
for all $\varphi \in W^{1,p(x)}_{0}(\Omega)$.
Such equations with $\varphi=u_{n}-u$ provide
$$
	\int_{\Omega}\bigl<|\nabla u_{n}|^{p(x)-2}\nabla u_{n}-|\nabla u|^{p(x)-2}\nabla u,\nabla (u_{n}-u)\bigl>\displaystyle=\int_{\Omega}\bigl[H(Tv_{n})-H(Tv)\bigl](u_{n}-u).
$$

The previous arguments implies that the sequence $(u_n)$ is bounded in $L^{p(x)}(\Omega)$. Thus by H\"{o}lder inequality we have
$$ \left| \int_{\Omega}\bigl[H(Tv_{n})-H(Tv)\bigl](u_{n}-u)  \right| \leq C|(H o T)(v_n) - (H o T)(v)|_{L^{p^{'}(x)}} $$
where the constant $C$ does not depend on $n \in \mathbb{N}.$ Since $H o T$ is continuous we have
$$\int_{\Omega}\bigl<|\nabla u_{n}|^{p(x)-2}\nabla u_{n}-|\nabla u|^{p(x)-2}\nabla u,\nabla (u_{n}-u)\bigl>  \rightarrow 0$$
which implies the continuity of $S.$

We claim that there exists $R>0$ such that if $u = \theta S(u)$ with $\theta \in [0,1]$ then $|u|_{L^{p(x)}} < R.$ In fact, if $\theta =0$ then $u=0.$ Suppose that $\theta \neq 0$. In this case we have $S(u) = \frac{u}{\theta}$ and such equality implies the identity
$$\int_{\Omega}\bigl|\nabla \Big(\frac{u}{\theta}\Big)\bigl|^{p(x)-2}\nabla\Big(\frac{u}{\theta}\Big)\nabla\varphi=\int_{\Omega}H(Tu)\varphi  $$
for all $ \varphi \in W^{1,p(x)}_{0}(\Omega)$. Using the test function $\varphi = \frac{u}{\theta}$, the inequality \eqref{eq1.3} and the embedding $L^{p(x)}(\Omega) \hookrightarrow L^{1}(\Omega)$ we get
$$\int_{\Omega} \left|\nabla \left(\frac{u}{\theta}\right)\right|^{p(x)}  \leq K_0 \int_{\Omega} \frac{|u|}{\theta} \leq \frac{C}{\theta} |u|_{L^{p(x)}(\Omega)}$$
where $C>0$ is a constant that does not depend on $u$ and $\theta$. If $|\nabla u|_{L^{p(x)}} >1$ we have by the Poincar\'{e} inequality and Proposition \ref{norm-property} that $| u|^{p_{-} -1}_{L^{p(x)}} \leq \theta ^{p^{-}-1}C$ where $C$ is a constant that does not depend on $u$  and $\theta.$

Since $\theta \in (0,1]$ by Schaefer's fixed Point Theorem there exists $u \in L^{p(x)}(\Omega)$ such that $u = S(u).$ Thus
\begin{equation}\label{eq1.4}
\int_{\Omega}|\nabla u|^{p(x)-2}\nabla u\nabla\varphi  =\int_{\Omega}\biggl(\frac{f(x,Tu)|Tu|_{L^{q(x)}}^{\alpha(x)}}{\mathcal{A}(x,|Tu|_{L^{r(x)}})}+\frac{g(x,Tu)|Tu|_{L^{s(x)}}^{\gamma(x)}}{\mathcal{A}(x,|Tu|_{L^{r(x)}})}\biggl)\varphi 
\end{equation}
for all $ \varphi \in W_{0}^{1,p(x)}(\Omega).$

We claim that $u \in [\underline{u}, \overline{u}].$ Considering $w = Tu$ in \eqref{eq1.1} and subtracting  from \eqref{eq1.4} we get
\begin{small}
\begin{align*}
\int_{\Omega}\bigl<|\nabla \underline{u}|^{p(x)-2}\nabla\underline{u}-|\nabla u|^{p(x)-2}\nabla u,\nabla\varphi\bigl>\displaystyle&\leq\int_{\Omega}\Big(\frac{f(x,\underline{u})|\underline{u}|_{L^{q(x)}}^{\alpha(x)}-f(x,Tu)|Tu|_{L^{q(x)}}^{\alpha(x)}}{\mathcal{A}(x,|Tu|_{L^{r(x)}})}\Big)\varphi\\
&+\int_{\Omega}\Big(\frac{g(x,\underline{u})|\underline{u}|_{L^{s(x)}}^{\gamma(x)}-g(x,Tu)|Tu|_{L^{s(x)}}^{\gamma(x)}}{\mathcal{A}(x,|Tu|_{L^{r(x)}})}\Big)\varphi
\end{align*}
\end{small}
for all $\varphi\in W_{0}^{1,p(x)}(\Omega)$ with $\varphi\geq0.$

Using the test function $\varphi := (\underline{u}-u)_{+} = \max\{\underline{u} - u,0\}$ and using  that $f,g\geq0$ in $[0,|\overline{u}|_{L^{\infty}}],$ $Tu=\underline{u}$ in $\bigl\{\underline{u}\geq u\bigl\}:=\bigl\{x\in\Omega:\underline{u}(x)\geq u(x)\bigl\}$  we get
\begin{small}
\begin{align*}
	\int_{\{\underline{u}\geq u\}}\bigl<|\nabla \underline{u}|^{p(x)-2}\nabla\underline{u}-|\nabla u|^{p(x)-2}\nabla u,\nabla(\underline{u}-u)\bigl>\displaystyle&\leq \int_{\{\underline{u}\geq u\}}\frac{f(x,\underline{u})
		(|\underline{u}|_{L^{q(x)}}^{\alpha(x)}-|Tu|_{L^{q(x)}}^{\alpha(x)})}{\mathcal{A}(x,|Tu|_{L^{r(x)}})}\varphi\\
	&+\int_{\{\underline{u}\geq u\}}\frac{g(x,\underline{u})
		(|\underline{u}|_{L^{s(x)}}^{\gamma(x)}-|Tu|_{L^{s(x)}}^{\gamma(x)})}{\mathcal{A}(x,|Tu|_{L^{r(x)}})}\varphi \\
	& \leq 0
\end{align*}
\end{small}
which implies that $\underline{u} \leq u$. A similar reasoning provides the inequality $u \leq \overline{u}.$
\end{proof}

\section{Applications}

The main goal of this section is to apply Theorem \ref{theorem-to-(P)} in some classes of nonlocal problems.

\subsection{A sublinear problem:}In this section we use Theorem \ref{theorem-to-(P)} to study the nonlocal problem
$$
\left \{
\begin{array}{rclcl}
-\mathcal{A}(x,|u|_{L^{r(x)}})\Delta_{p(x)}u &=&u^{\beta(x)}|u|_{L^{q(x)}}^{\alpha(x)} \ \mbox{in} \   \Omega, \\
u&=& 0\ \mbox{on}  \ \partial \Omega
\end{array}
\right.\leqno{(Ps)}
$$

The above problem in the case $p(x)\equiv2$  was considered  recently in  \cite{gel-giovany}. The  result of this section generalizes \cite[Theorem 3]{gel-giovany}.

\begin{theorem}\label{teo-sublinear} Suppose that $r,p,q,\alpha$ satisfy $(H_{0})$ and let  $\beta\in L^{\infty}(\Omega)$ be a nonnegative function. Consider also that 
	$\alpha^{+}+\beta^{+}<p^{-}-1$. Let  $a_{0}>0$ be a positive constant. Suppose that one of the conditions hold
\vspace{0.2cm}		
	
		\noindent{\bf $(A_{1})$}$\mathcal{A}(x,t)\geq a_{0}\;\text{in}\;\overline{\Omega}\times[0,\infty), $

\vspace{0.2cm}

\noindent{\bf $(A_{2})$} $0<\mathcal{A}(x,t)\leq a_{0}\;\text{in}\;\overline{\Omega}\times(0,\infty)$ and $\lim_{t \rightarrow + \infty}\mathcal{A}(x,t)=a_{\infty}>0$ uniformly in $\Omega.$ 

\vspace{0.2cm}	

Then $(Ps)$ has a positive solution.
\end{theorem}

\begin{proof} Suppose that $(A_1)$  holds, that is , $\mathcal{A}(x,t)\geq a_{0} $ in $\overline{\Omega} \times [0,+\infty)$. We will start by constructing  $\overline{u}$. Let $\lambda>0$ and consider  $z_{\lambda}\in W_{0}^{1,p(x)}(\Omega)\cap L^{\infty}(\Omega)$ the unique  solution of \eqref{probl-linear-lambda} where $\lambda$ will be chosen later. 
	
For $\lambda >0$ large by Lemma \ref{Fan} there is a constant $K>1$ that does not depend on $\lambda$  such that
\begin{equation}\label{desig1-p-supsol}
0<z_{\lambda}(x)\leq K\lambda^{\frac{1}{p^{-}-1}}\;\text{in}\;\Omega.
\end{equation}	
	
Since $\alpha^{+}+\beta^{+}<p^{-}-1$ we can choose $\lambda>1$ such that \eqref{desig1-p-supsol} occurs and
\begin{equation}\label{desig2-p-supsol}
\frac{1}{a_{0}}K^{\beta^{+}}\lambda^{\frac{\alpha^{+}+\beta^{+}}{p^{-}-1}}\max\{|K|_{L^{q(x)}}^{\alpha^{-}},|K|_{L^{q(x)}}^{\alpha^{+}}\}\leq\lambda.
\end{equation}	

By \eqref{desig1-p-supsol} and \eqref{desig2-p-supsol} we get $$\frac{1}{a_{0}}z_{\lambda}^{\beta(x)}|z_{\lambda}|_{L^{q(x)}}^{\alpha(x)}\leq   \lambda. $$
Therefore
\begin{equation*}
\begin{aligned}
\left\{\begin{array}{rcl}
-\Delta_{p(x)}z_{\lambda}&\geq&\dfrac{1}{\mathcal{A}(x,|w|_{L^{r(x)}})}z_{\lambda}^{\beta(x)}|z_{\lambda}|_{L^{q(x)}}^{\alpha(x)}\;\;\mbox{in}\;\;\Omega,\\
\vspace{.2cm}
z_{\lambda}&=&0\;\;\mbox{on}\;\;\partial\Omega
\end{array}
\right.
\end{aligned}
\end{equation*}
for all  $w\in L^{\infty}({\Omega}).$

Define $\mathcal{A}_{\lambda}:=\max\bigl\{\mathcal{A}(x,t):(x,t)\in\overline{\Omega}\times\bigl[0,|z_{\lambda}|_{L^{r(x)}}\bigl]\bigl\}.$ We have $$a_{0}\leq\mathcal{A}(x,|w|_{L^{r(x)}})\leq \mathcal{A}_{\lambda}\;\;\mbox{in}\;\Omega $$
for all $w \in [0,z_{\lambda}].$

Now we will construct $\underline{u}.$ Since $\partial \Omega$ is $C^2$ there is a constant $\delta >0$ such that $d \in C^{2}(\overline{\partial \Omega_{3 \delta}})$ and $|\nabla d(x)| \equiv 1$ where $d(x):= dist(x,\partial \Omega)$ and $\overline{\partial \Omega_{3 \delta}}:=\{x \in  \overline{\Omega}; d(x) \leq 3 \delta\}$.
From \cite[Page 12]{Liu} we have that  for  $\sigma \in (0, \delta)$ small  the function $\phi=\phi(k,\sigma)$ defined by


\begin{equation*}
\phi(x)=\left\{\begin{array}{lcl}
e^{kd(x)}-1 & \text{ if } & d(x)<\sigma,\\ 
e^{k\sigma}-1+\int_{\sigma}^{d(x)}ke^{k\sigma}\Big(\frac{2\delta-t}{2\delta-\sigma}\Big)^{\frac{2}{p^{-}-1}}dt & \text{ if } &\sigma\leq d(x)<2\delta,\\
e^{k\sigma}-1+\int_{\sigma}^{2\delta}ke^{k\sigma}\Big(\frac{2\delta-t}{2\delta-\sigma}\Big)^{\frac{2}{p^{-}-1}}dt & \text{ if } & 2\delta \leq d(x)
\end{array}
\right.
\end{equation*}


belongs to  $ C^{1}_{0}(\overline{\Omega})$ where $k>0$ is an arbitrary number. They also proved that
$$-\Delta_{p(x)}(\mu\phi)=\begin{cases}
-k(k\mu e^{kd(x)})^{p(x)\!-1}\Big[(p(x)\!\!-1)+(d(x)\!\!+\frac{\ln k\mu}{k})\nabla p(x)\nabla d(x)\\
+\frac{\Delta d(x)}{k}\Big] \;\; \mbox{ if}\quad d(x)<\sigma,\\
\Big\{\frac{1}{2\delta-\sigma}\frac{2(p(x)-1)}{p^{-}-1}\!-\!\Big(\frac{2\delta-d(x)}{2\delta-\sigma}\Big)\Big[\ln k\mu e^{k\sigma}\Big(\frac{2\delta-d(x)}{2\delta-\sigma}\Big)^{\frac{2}{p^{-}-1}}\nabla p(x)\nabla d(x)\\
+\Delta d(x)\Big]\Big\}
(k\mu e^{k\sigma})^{p(x)-1}\Big(\frac{2\delta-d(x)}{2\delta-\sigma}\Big)^{\frac{2(p(x)-1)}{p^{-}-1}-1}\;\; \mbox{ if}\quad \sigma < d(x)<2\delta,\\
0\;\; \mbox{ if}\quad 2\delta<d(x)
\end{cases}
$$
for all $\mu >0$.

Let $\sigma=\frac{1}{k}\ln 2^{\frac{1}{p^{+}}}$ and $\mu=e^{-ak}$ where $a=\frac{p^{-}-1}{\max_{\overline{\Omega}}|\nabla p|+1}.$ Then $e^{k\sigma}=2^{\frac{1}{p^{+}}}$ and $k\mu\leq1$ if $k>0$ is large. From  \cite[Page 12]{Liu} we have
$$-\Delta_{p(x)}(\mu\phi)\leq0<\frac{1}{\mathcal{A}_{\lambda}}(\mu\phi)^{\beta(x)}|\mu\phi|_{L^{q(x)}}^{\alpha(x)} \;\text{ if }\;d(x)<\delta\;\text{or}\;\;2\delta<d(x).$$
and
\begin{equation}\label{desi1-p-subsol}
-\Delta_{p(x)}(\mu\phi)\leq \tilde{C}(k\mu)^{p^{-}-1}|\ln k\mu| \;\text{ if }\;\sigma < d(x)<2\delta.
\end{equation}	

We claim that 
\begin{equation}\label{lhopital}
\displaystyle\lim_{k \rightarrow + \infty} \displaystyle\frac{\tilde{C} k^{p_{-}-1}}{e^{ak(p^{-}-1 - (\alpha^{+}  + \beta^{+}))}} \left| \ln \displaystyle\frac{k}{e^{ak}}\right| = 0.
\end{equation}
In fact, note that 
$$\displaystyle\lim_{k \rightarrow + \infty} \displaystyle\frac{\tilde{C} k^{p_{-}-1}}{e^{ak(p^{-}-1 - (\alpha^{+}  + \beta^{+}))}} \left| \ln \displaystyle\frac{k}{e^{ak}}\right| = \displaystyle\lim_{k \rightarrow + \infty} \displaystyle\frac{f(k)}{g(k)} $$
where  
$$f(k) = \left| \ln \frac{k}{e^{ak}}\right| \ \text{and} \  g(k) = \displaystyle\frac{e^{ak(p^{-} -1 - (\alpha^{+} + \beta^{+}))}}{\tilde{C}k^{p^{-}-1}}, k >0.$$
Note that $f^{'}(k) = a - \frac{1}{k}$ for $k>0$ large which implies that $\displaystyle\lim_{k \rightarrow + \infty }f^{'}(k)  =a.$ Observe that

$$ g^{'}(k)= \displaystyle\frac{e^{ak(p^{-} -1 - (\alpha^{+} + \beta^{+}))}}{\tilde{C} k^{p^{-}-1}} \left( a(p^{-} -1 - (\alpha^{+} + \beta^{+}) - \displaystyle\frac{p^{-}-1}{k})\right)$$
and note also that $\displaystyle\lim_{k \rightarrow + \infty }g^{'}(k)  = + \infty$ because $\alpha^{+} + \beta^{+}  < p^{-} -1.$ Thus by L'Hospital's rule we have the claim.

If $\sigma\leq d(x)<2\delta$    we have $\phi(x)\geq 2^{\frac{1}{p^{+}}}-1$ for all $k>0$ because $e^{k \sigma} = 2^{\frac{1}{p^{+}}}$.Thus, there is a constant $C_{0}>0$ that does not depend on $k$ such that $|\phi|_{L^{q(x)}}^{\alpha(x)}\geq C_{0}\;\text{if}\;\sigma\leq d(x)<2\delta.$ By \eqref{lhopital} we can choose  $k>0$ large enough such that
\begin{equation}\label{desig2-p-subsol}
\frac{C_{1}k^{p^{-}-1}}{e^{ak[(p^{-}-1)-(\alpha^{+}+\beta^{+})]}}\Big|\ln\frac{k}{e^{ak}}\Big|\leq\frac{C_{0}}{\mathcal{A}_{\lambda}}(2^{\frac{1}{p^{+}}}-1)^{\beta^{+}}.
\end{equation}

It is possible to choose $k>0 $ large such that $\mu \phi (x) \leq 1$ for all $x \in \Omega$ that satisfies $\sigma < d(x) < \delta.$ Therefore from \eqref{desi1-p-subsol} and \eqref{desig2-p-subsol}  we have
$$-\Delta_{p(x)}(\mu\phi)\leq \frac{1}{\mathcal{A}_{\lambda}}(\mu\phi)^{\beta(x)}|\mu\phi|_{L^{q(x)}}^{\alpha(x)}\;\;\text{if}\;\;\sigma <d(x)<2\delta$$
for $k>0 $ large enough.
Fix $k>0$ satisfying the above property and the inequality $- \Delta_{p(x)} (\mu \phi) \leq 1.$ For $\lambda >1 $ we have $- \Delta_{p(x)} (\mu \phi) \leq  - \Delta_{p(x)} z_{\lambda}$. Therefore $\mu \phi \leq z_{\lambda}.$ The first part of the result is proved.

Now suppose that $0<\mathcal{A}(x,t)\leq a_{0}$ in $\overline{\Omega}\times(0,\infty).$ Let $\delta, \sigma, \mu, a, \lambda, z_{\lambda}$ and $\phi$ as before. From the previous arguments there exist $k>0$ large enough and  $\mu>0$ small such that
$$-\Delta_{p(x)}(\mu\phi)\leq1 \;\;\text{ and}\;\;-\Delta_{p(x)}(\mu\phi)\leq \frac{1}{a_{0}}(\mu\phi)^{\beta(x)}|\mu\phi|_{L^{q(x)}}^{\alpha(x)}\;\;\text{in}\;\Omega.$$
In particular for $ w\in L^{\infty}(\Omega)$ with $\mu\phi\leq w$  we have
\begin{equation}\label{aux1}
-\Delta_{p(x)}(\mu\phi)\leq \frac{1}{\mathcal{A}(x,|w|_{L^{r(x)}})}(\mu\phi)^{\beta(x)}|\mu\phi|_{L^{q(x)}}^{\alpha(x)}\;\;\text{in}\;\Omega.
\end{equation}

Since $\lim_{t\rightarrow\infty}\mathcal{A}(x,t)=a_{\infty}>0$ uniformly in $\Omega$ there is a constant  $a_{1}>0$ such that $\mathcal{A}(x,t)\geq\frac{a_{\infty}}{2}$ em $\overline{\Omega}\times(a_{1},\infty).$ Let $m_{k}:=\min\big\{\mathcal{A}(x,t):\overline{\Omega}\times[|\mu\phi|_{L^{r(x)}}, a_{1}] \big\}>0$ and $\mathcal{A}_{k}:=\min\big\{m_k,\frac{a_{\infty}}{2}\big\}$ then 
$\mathcal{A}(x,t)\geq\mathcal{A}_{k}\; \text{in}\; \overline{\Omega}\times[|\mu\phi|_{L^{r(x)}},\infty).$

Fix $k>0$ satisfying \eqref{aux1}. Let $\lambda>1$ such that  \eqref{desig1-p-supsol} ocurrs and $$\frac{1}{\mathcal{A}_{k}}K^{\beta^{+}}\lambda^{\frac{\alpha^{+}+\beta^{+}}{p^{-}-1}}\max\{|K|_{L^{q(x)}}^{\alpha^{-}},|K|_{L^{q(x)}}^{\alpha^{+}}\}\leq\lambda$$
where $K>1$ is a constant that does not depend on $k$ and $\lambda$ (see Lemma \ref{Fan}).  Thus for all $w\in[\mu\phi,z_{\lambda}]$ we have
$$-\Delta_{p(x)}z_{\lambda}\leq \frac{1}{\mathcal{A}(x,|w|_{L^{r(x)}})}z_{\lambda}^{\beta(x)}|z_{\lambda}|_{L^{q(x)}}^{\alpha(x)}\; \text{in}\; \Omega.$$
From the weak comparison principle we have $\mu \phi \leq z_{\lambda}.$ Therefore $(\mu \phi, z_{\lambda})$ is a sub-supersolution pair for $(P_s).$	
\end{proof}

\subsection{A concave-convex problem:}

In this section we consider the following nonlocal problem with concave-convex nonlinearities
$$
\left\{\begin{array}{rcl}
-\mathcal{A}(x,|u|_{L^{r(x)}})\Delta_{p(x)} u &=&\lambda |u|^{\beta(x)-1}u|u|_{L^{q(x)}}^{\alpha(x)}+\theta |u|^{\eta(x)-1}u|u|_{L^{s(x)}}^{\gamma(x)} \;\;\mbox{in}\;\;\Omega,\\
\vspace{.2cm}
u&=&0\;\;\mbox{on}\;\;\partial\Omega.
\end{array}
\right.\leqno{(P)_{\lambda,\theta}}
$$

The local version of $(P)_{\lambda,\theta}$ with $p(x)\equiv2$  and constant exponents was considered in the famous paper by Ambrosetti-Brezis-Cerami \cite{abc} by using a sub-super argument. In \cite{gel-giovany} the problem $(P)_{\lambda,\theta}$ was studied with $p(x)\equiv2.$ The following result generalizes \cite[Theorem 4]{gel-giovany}.

\begin{theorem}Suppose that $r,p,q,s,\alpha$ and $\gamma$ satisfy $(H_{0})$ and  $\beta,\eta\in L^{\infty}(\Omega)$  are nonnegative functions with $0<\alpha^{-}+\beta^{-}\leq\alpha^{+}+\beta^{+}<p^{-}-1$. Let $a_0,b_0>0$ positive numbers. The following assertions hold
	
\vspace{0.2cm}	

	\noindent{\bf $(A_{1})$} If $p^{+}-1<\eta^{-}+\gamma^{-}$ and $\mathcal{A}(x,t)\geq a_{0}\;\text{in}\;\overline{\Omega}\times [0, b_{0}]$ then given $\theta>0$ there exists $\lambda_{0}>0$ such that for each $\lambda\in(0,\lambda_{0})$ the  problem $(P)_{\lambda,\theta}$ has a positive solution $u_{\lambda,\theta}.$
	
\vspace{0.2cm}
	
	\noindent{\bf $(A_{2})$} If $p^{-}-1<\eta^{+}+\gamma^{+}$ and $0<\mathcal{A}(x,t)\leq a_{0}\;\mbox{in}\;\; \overline{\Omega}\times(0,\infty)$ and $\lim_{t\rightarrow\infty}\mathcal{A}(x,t)=b_{0}\;\text{uniformly in }\;\overline{\Omega}$
	then given $\lambda>0$ there exists $\theta_{0}>0$ such that for each $\theta\in(0,\theta_{0})$ the problem  $(P)_{\lambda,\theta}$ has a positive solution $u_{\lambda,\theta}.$
\end{theorem}

\begin{proof} Suppose that $(A_{1})$ occurs. Let $z_{\lambda}\in W_{0}^{1,p(x)}({\Omega})\bigcap L^{\infty}(\Omega)$ be the unique solution of \eqref{probl-linear-lambda} where $\lambda\in(0,1)$ will be chosen before.

Lemma \ref{Fan} implies that for $\lambda>0$ small enough there exists a constant $K>1$ that does not depend on $\lambda$ such that  
\begin{equation}\label{desig1-p-supsol-concavo}
0<z_{\lambda}(x)\leq K\lambda^{\frac{1}{p^{+}-1}}\;\text{in}\;\Omega.
\end{equation}

Let $\overline{K}:=\max\big\{|K|_{L^{q(x)}}^{\alpha^{+}},|K|_{L^{q(x)}}^{\alpha^{-}},|K|_{L^{s(x)}}^{\gamma^{+}},|K|_{L^{s(x)}}^{\gamma^{-}}\big\}.$  For each $\theta >0$ we can choose $0<\lambda_0<1 $ small enough, depending on $\theta,$  such that the inequalities
\begin{equation*}
\lambda\geq\dfrac{1}{a_{0}}\left(\lambda^{\frac{p^{+}-1+\beta^{-}+\alpha^{-}}{p^{+}-1}}K^{\beta^{+}}\overline{K} + \theta \lambda^{\frac{\eta^{-}+\gamma^{-}}{p^{+}-1}}K^{\eta^{+}}\overline{K}\right), \lambda \in (0,\lambda_0)
\end{equation*}
and \eqref{desig1-p-supsol-concavo} hold because $\alpha^{-} + \beta^{-} >0   $  and $p^{+} -1 < \eta^{-} + \gamma^{-}.$

There is  $\lambda_0>0$ small such that
\begin{align*}
\frac{1}{a_0}(\lambda z_{\lambda}^{\beta(x)} |z_\lambda|^{\alpha(x)}_{L^{q(x)}} + \theta z^{\eta(x)}_{\lambda}|z_\lambda|^{\gamma(x)}_{L^{s(x)}} )&\leq \lambda (K \lambda^{\frac{1}{p^{+}-1}})^{\beta(x)} |K \lambda^{\frac{1}{p^{+}-1}}|^{\alpha(x)}_{L^{q(x)}} \\
&+\theta (K \lambda^{\frac{1}{p^{+}-1}})^{\eta(x)}|K \lambda^{\frac{1}{p^{+}-1}}|^{\gamma (x)}_{L^{s(x)}} \\
&\leq \lambda.
\end{align*}
for all $\lambda \in (0,\lambda_0).$
Thus for $\lambda \in (0,\lambda_0)$ we get	
	
$$\frac{1}{a_0}(\lambda z_{\lambda}^{\beta(x)} |z_\lambda|^{\alpha(x)}_{L^{q(x)}} + \theta z^{\eta(x)}_{\lambda}|z_\lambda|^{\gamma(x)}_{L^{s(x)}}) \leq \lambda.	$$
If necessary consider a smaller value for $\lambda_0$ such that $|z_{\lambda}|_{L^{r(x)}} \leq |K|_{L^{r(x)}} \lambda^{\frac{1}{p^{+}-1}} \leq b_0.$ Thus for all $w \in [0, |z_\lambda|_{L^{r(x)}}]$ we have $\mathcal{A}(x,|w|_{L^{r(x)}}) \geq a_0.$ Therefore

\begin{equation}\label{subsol-p-concavo}
-\Delta_{p(x)}z_{\lambda}\geq\frac{1}{\mathcal{A}(x,|w|_{L^{r(x)}})}\left(\lambda z_{\lambda}^{\beta(x)}|z_{\lambda}|_{L^{q(x)}}^{\alpha(x)}+\theta z_{\lambda}^{\eta(x)}|z_{\lambda}|_{L^{s(x)}}^{\gamma(x)}\right)\;\mbox{in}\;\Omega\\
\end{equation}
for $\lambda \in (0,\lambda_0).$

Now consider $\phi, \delta,\sigma,\mu$ and $a$ as in the proof of Theorem \ref{teo-sublinear}. Fix $\lambda\in(0,\lambda_{0})$  such that \eqref{subsol-p-concavo} holds. Let $\mathcal{A}_{0}:=\max\big\{\mathcal{A}(x,t):(x,t)\in\overline{\Omega}\times[0,b_{0}]\big\}.$

Since $\alpha^{+}+\beta^{+}<p^{-}-1$  the arguments of the proof of Theorem \ref{teo-sublinear} implies that if $\mu=\mu(\lambda)>0$ is small enough then
$$-\Delta_{p(x)}(\mu\phi)\leq\lambda \ \text{in} \ \Omega$$
and
\begin{align*}
-\Delta_{p(x)}(\mu\phi)&\leq\frac{1}{\mathcal{A}_{0}}\lambda(\mu\phi)^{\beta(x)}|\mu\phi|_{L^{q(x)}}^{\alpha(x)}   \\
&\leq  \frac{1}{\mathcal{A}(x,|w|_{}L^{r(x)})}\lambda(\mu\phi)^{\beta(x)}|\mu\phi|_{L^{q(x)}}^{\alpha(x)}
\end{align*}
in $\Omega$ for all $w \in [0, |z_{\lambda}|_{L^{r(x)}}].$ The weak comparison principle implies that $\mu \phi \leq z_{\lambda}$  for $\mu >0$ small enough. Therefore $(\mu \phi, z_{\lambda})$ is a sub-super solution pair for $(P)_{\lambda, \theta}.$

Now we will prove the theorem in the second case. Consider again $\phi,\delta,\sigma,\mu$ and $a$ as in the proof of Theorem \ref{teo-sublinear}. Let $\lambda \in(0,\infty)$. Since $\alpha^{+}+\beta^{+}<p^{-}-1$  we can repeat the arguments of Theorem \ref{teo-sublinear} to obtain $\mu=\mu(\lambda)>0$ small depending only on $ \lambda$  such that

$$-\Delta_{p(x)}(\mu\phi)\leq 1\;\;\text{and}\;\;-\Delta_{p(x)}(\mu\phi)\leq\frac{\lambda}{a_{0}}(\mu\phi)^{\beta(x)}|\mu\phi|_{L^{q(x)}}^{\alpha(x)}\;\;\text{in}\;\Omega.$$

Let $z_{M}\in W_{0}^{1,p(x)}(\Omega)\cap L^{\infty}(\Omega)$ the unique solution of $\eqref{probl-linear-lambda}$  where  $M>0$ will be chosen later.

For $M\geq 1$ large enough there is a constant $K>1$ that does not depend on  $M$  such that
\begin{equation}\label{desig2-p-supsol-concavo}
0<z_{M}(x)\leq KM^{\frac{1}{p^{-}-1}}\;\text{in}\;\Omega.
\end{equation}

We want to obtain $M>1$ such that for each $w\in L^{\infty}(\Omega)$ with $\mu\phi\leq w$ the inequality
\begin{equation}\label{eqnova}
M\geq\frac{1}{\mathcal{A}(x,|w|_{L^{r(x)}})}\left(\lambda z_{M}^{\beta(x)}|z_{M}|_{L^{q(x)}}^{\alpha(x)}+\theta z_{M}^{\eta(x)}|z_{M}|_{L^{s(x)}}^{\gamma(x)}\right)\;\mbox{in}\;\Omega
\end{equation}
occurs.

Since $\mathcal{A}$ is continuous and $\displaystyle\lim_{t \rightarrow + \infty } \mathcal{A}(x,t) = b_0 >0$ uniformly in $\Omega$ there is a constant $a_1 >0$ such that $\mathcal{A}(x,t) \geq \frac{b_0}{2}$ in $\overline{\Omega} \times (a_1, + \infty)$. Consider 
$$m_{\lambda} = \min \{ \mathcal{A}(x,t): (x,t) \in \overline{\Omega} \times [|\mu \phi|_{L^{r(x)}}, a_1]\} >0$$
and $\mathcal{A}_{\lambda} = \min \{m_{\lambda}, \frac{b_0}{2}\}$. Then $\mathcal{A}(x,t) \geq \mathcal{A}_{\lambda}$ in $\overline{\Omega} \times [|\mu \phi|_{L^{r(x)}}, + \infty).$ Thus there exists a constant $\mathcal{A}_{\lambda} > 0 $ with $\mathcal{A}_{\lambda} \leq \mathcal{A} (x, |w|_{L^{r(x)}}) \leq a_0$ for all $w \in L^{\infty}(\Omega)$ with $\mu \phi \leq w$.

By \eqref{desig2-p-supsol-concavo} we have
\begin{equation}\label{equiv-rel}
\frac{(\lambda z^{\beta (x)}_{M} |z_{M}|^{\alpha (x)}_{L^{q(x)}} + \theta z^{\eta (x)}_{M} |z_{M}|^{\gamma (x)}_{L^{s(x)}})}{\mathcal{A} (x, |w|_{L^{r(x)}})} \leq \frac{(\lambda M^{\frac{\beta^{+} + \alpha^{+}}{p^{-}-1}}\overline{C} + \theta \overline{C} M^{\frac{\eta^{+} + \gamma^{+}}{p^{-}-1}} )}{\mathcal{A}_{\lambda}}
\end{equation}
with $\overline{C}= \max\{K^{\beta^+} \overline{K}, K^{\eta^+} \overline{K}\}$ where $\overline{K} = \max\{|K|^{\alpha^+}_{L^{q(x)}}, |K|^{\alpha^-}_{L^{q(x)}}, |K|^{\gamma^+}_{L^{s(x)}},|K|^{\gamma^-}_{L^{s(x)}}\}$.

Denoting by $I$  the right-hand side of \eqref{equiv-rel} we have $I \leq M$ if and only if

\begin{equation}\label{equiv-rel2}
1\geq\dfrac{1}{\mathcal{A}_{\lambda}}\left(\lambda \overline{C} M^{\frac{\beta^{+}+\alpha^{+}}{p^{-}-1}-1}+\theta \overline{C} M^{\frac{\eta^{+}+\gamma^{+}}{p^{-}-1}-1}\right).
\end{equation}
Since $\alpha^{+}+\beta^{+}<p^{-}-1<\eta^{+}+\gamma^{+}$ the function
$$\Psi(t) = \frac{\lambda \overline{C}t^{\frac{\alpha^+ + \beta^+}{p^{-}-1}-1} + \theta \overline{C}t^{\frac{\eta^+ + \gamma^+}{p^- -1}-1}}{\mathcal{A}_{\lambda}}  , t >0.$$
belongs to $C^{1}\big((0,\infty),\mathbb{R}\big)$ and attains a global minimum at

\begin{equation}\label{minimum}
M_{\lambda,\theta}:=M(\lambda,\theta)=L\Biggl(\dfrac{\lambda}{\theta}\Biggl)^{\frac{p^{-}-1}{(\eta^{+}+\gamma^{+})-(\beta^{+}+\alpha^{+})}}
\end{equation}
where $L=\left(\frac{(p^{-}-1)-(\beta^{+}+\alpha^{+})}{(\eta^{+}+\gamma^{+})-(p^{-}-1)}\right)^{\frac{p^{-}-1}{(\eta^{+}+\gamma^{+})-(\beta^{+}+\alpha^{+})}}.$
The inequality \eqref{equiv-rel2} is equivalent to find $M_{\lambda,\theta}>0$ such that $\Psi(M_{\lambda,\theta}) \leq 1.$ By \eqref{minimum} we have $\Psi(M_{\lambda,\theta}) \leq 1$ if and only if
$$\frac{\lambda \overline{C}P}{\mathcal{A}_{\lambda}}   \left( \frac{\lambda}{\theta}\right)^{\frac{ \alpha^{+} + \beta^{+}  - (p^{-} -1)}{\eta^{+} + \gamma^{+} - (\beta^{+}  +\alpha^{+})}} +\frac{\theta \overline{C}Q}{\mathcal{A}_{\lambda}}  \left( \frac{\lambda}{\theta}\right)^{\frac{\eta^{+} + \gamma^{+} - (p^{-} -1)}{\eta^{+} + \gamma^{+} - (\beta^{+}  +\alpha^{+})}} \leq 1$$
where $P=L^{\frac{\alpha^{+} \beta^{+} - (p^{-}-1)}{p^{-}-1}} $ and $Q = L^{\frac{\eta^{+} + \gamma^{+}}{p^{-}-1}}$. Notice that the above inequality holds if $\theta >0$ is small enough because $\alpha^{+}+\beta^{+}<p^{-}-1<\eta^{+}+\gamma^{+}$.
Thus for $\lambda >0$ fixed there exists $\theta_{0}= \theta_{0}(\lambda)$ such that for each $\theta \in (0,\theta_0)$ there is a number $M=M_{\lambda,\theta}>0$ such that \eqref{equiv-rel2} occurs. Consequently we have \eqref{eqnova}. Therefore
$$- \Delta_{p(x)} z_{M} \geq  \frac{1}{\mathcal{A}(x,|w|_{l^{r(x)}})} (\lambda z^{\beta(x)}_{M} |z_{M}|^{\alpha (x)}_{L^{q(x)}} + \theta z^{\eta(x)}_{M} |z_{M}|^{\gamma (x)}_{L^{s(x)}})  \ \text{in} \ \Omega.$$
Considering if necessary  a smaller $\theta_0 > 0$ we get $ M \geq 1$ . Therefore $-\Delta_{p(x)}(\mu \phi) \leq -\Delta_{p(x)}z_{M}$ in $\Omega.$  The weak comparison principle implies that $\mu \phi \leq z_M$. Then $(\mu \phi, z_M)$ is a sub-supersolution pair for $(P)_{\lambda,\theta}.$ The proof is finished.
\end{proof}

\subsection{A generalization of the logistic equation:}

In the previous sections we considered at least one of the conditions $\mathcal{A}(x,t)\geq a_{0}>0$ or $0<\mathcal{A}(x,t)\leq a_{\infty}, t>0.$ In this last section we study a generalization of the classic logistic equation where the function $\mathcal{A}(x,t)$  can satisfy
$$\mathcal{A}(x,0)\geq0,\;\;\;\lim_{t\rightarrow0^{+}}\mathcal{A}(x,t)=\infty
\;\;\;\mbox{and}\;\;\;\lim_{t\rightarrow + \infty}\mathcal{A}(x,t)=\pm\infty.$$

We will attack the problem
$$\left\{\begin{array}{rcl}
-\mathcal{A}(x,|u|_{L^{r(x)}})\Delta_{p(x)} u&=&\lambda f(u)|u|_{L^{q(x)}}^{\alpha(x)}\;\;\mbox{in}\;\;\Omega,\\
\vspace{.2cm}
u&=&0\;\;\mbox{on}\;\;\partial\Omega.
\end{array}
\right. \leqno{(P)_{\lambda}}$$

We will suppose that there is a number $\theta>0$ such that the function $f:[0,\infty)\rightarrow\mathbb{R}$ satisfies the conditions:

\noindent{\bf $(f_{1})$} $\;f\in C^{0}([0,\theta],\mathbb{R}).$

\noindent{\bf $(f_{2})$} $\;f(0)=f(\theta)=0,\;\;f(t)>0\; \text{in}\;(0,\theta).$

Problem  $(P)_{\lambda}$ is a generalization of the problems studied in  \cite{chipot-correa, chipot-roy, gel-giovany}. The next result generalizes \cite[Theorem 5]{gel-giovany}.

\begin{theorem}
	Suppose that  $r,p,q,\alpha$ satisfy $(H_{0}).$ Consider also that $f$ satisfies $(f_{1}), (f_{2})$ and that	$\mathcal{A}(x,t)>0$ in $\overline{\Omega}\times\bigl(0,|\theta|_{L^{r(x)}}\bigl].$ Then there exists $\lambda_{0}>0$ such that $\lambda\geq\lambda_{0},$ $(P)_{\lambda}$ has a positive solution $u_{\lambda}\in[0,\theta].$
\end{theorem}
\begin{proof}
Consider the function $\widetilde{f}(t)=f(t)$ for $t\in[0,\theta]$ and $\widetilde{f}(t)=0$ for $t\in\mathbb{R}\setminus[0,\theta].$ The functional
$$J_{\lambda}(u)=\int_{\Omega}\frac{1}{p(x)}|\nabla u|^{p(x)}dx-\lambda\int_{\Omega}\widetilde{F}(u)dx,\;u\in W_{0}^{1,p(x)}(\Omega),$$
where $\widetilde{F}(t)=\int_{0}^{t}\widetilde{f}(s)ds$ is of class $C^{1}(W_{0}^{1,p(x)}(\Omega),\mathbb{R}).$  Since $|\widetilde{f}(t)|\leq C$ for $t\in\mathbb{R}$ we have that $J$ is coercive. Thus $J$ has a minimum $z_{\lambda}$ which is a weak solution of the problem

$$\left\{\begin{array}{rcl}
-\Delta_{p(x)} z&=&\lambda \widetilde{f}(z)\;\;\mbox{in}\;\;\Omega,\\
\vspace{.2cm}
z&=&0\;\;\mbox{on}\;\;\partial\Omega.
\end{array}
\right.$$

Consider a function $\varphi_{0}\in W_{0}^{1,p(x)}(\Omega)$ such that $\widetilde{F}(\varphi_{0})>0.$ Define $z_{0}:=z_{\widetilde{\lambda}_{0}}$ where $\widetilde{\lambda}_{0}>0$ satisfy 
$$\int_{\Omega}\frac{1}{p(x)}|\nabla \varphi_{0}|^{p(x)} <\widetilde{\lambda}_{0}\int_{\Omega}\widetilde{F}(\varphi_{0}).$$

Thus $J_{\widetilde{\lambda}_{0}}(z_{0})\leq J_{\widetilde{\lambda}_{0}}(\varphi_{0})<0.$ Since $J_{\widetilde{\lambda}_{0}}(0)=0$ we have $z_{0}\neq0.$  By \cite[Theorem 4.1]{fan-zao} we have $z_{0}\in W_{0}^{1,p(x)}(\Omega)\cap L^{\infty}(\Omega)$ and using \cite[Theorem 1.2]{fan-regular} we obtain that $z_{0}\in C^{1,\alpha}(\overline{\Omega}).$ Considering the test function $\varphi=z_{0}^{-}:=\min\{z_{0},0\}$ we get  $z_{0}=z_{0}^{+}\geq0.$ By Proposition  \ref{PC} we have $z_{0}>0.$

Considering the test function $\varphi=(z_{0}-\theta)^{+}\in W_{0}^{1,p(x)}(\Omega)$ we have
$$\int_{\Omega}|\nabla z_{0}|^{p(x)-2}\nabla z_{0}\nabla(z_{0}-\theta)^{+}=\widetilde{\lambda}_{0}\int_{\{z_{0}>\theta\}}\widetilde{f}(z_{0})(z_{0}-\theta)=0.$$
Therefore
\begin{eqnarray*}
	\int_{\{z_{0}>\theta\}}\bigl<|\nabla z_{0}|^{p(x)-2}\nabla z_{0}-|\nabla\theta|^{p(x)-2}\nabla \theta,\nabla(z_{0}-\theta)\bigl>=0
\end{eqnarray*}
which implies  $(z_{0}-\theta)_{+}=0$ in $\Omega.$ Thus $0<z_{0}\leq\theta.$

Note that there is a constant $C>0$ such that $|z_{0}|_{L^{q(x)}}^{\alpha(x)}\geq C.$ Define $\mathcal{A}_{0}:=\max\big\{\mathcal{A}(x,t):(x,t)\in\overline{\Omega}\times[|z_{0}|_{L^{r(x)}},|\theta|_{L^{r(x)}}]\big\}$ and $\mu_{0}=\frac{\mathcal{A}_{0}}{C}.$ Then we have
$$-\Delta_{p(x)}z_{0}=\widetilde{\lambda}_{0} f(z_{0})=\frac{1}{\mathcal{A}_{0}}\widetilde{\lambda}_{0}\mu_{0}f(z_{0})|z_{0}|_{L^{q(x)}}^{\alpha(x)}\frac{\mathcal{A}_{0}}{\mu_{0}|z_{0}|_{L^{q(x)}}^{\alpha(x)}}\leq \frac{1}{\mathcal{A}_{0}}\widetilde{\lambda}_{0}\mu_{0}f(z_{0})|z_{0}|_{L^{q(x)}}^{\alpha(x)}.$$

Thus for each $\lambda\geq\widetilde{\lambda}_{0}\mu_{0}$ and $w\in[\varphi,\theta]$ we get
$$-\Delta_{p(x)}z_{0}\leq\frac{1}{\mathcal{A}(x,|w|_{L^{r(x)}})}\lambda f(z_{0})|z_{0}|_{L^{q(x)}}^{\alpha(x)}.$$

Since $f(\theta)=0$ it follows that  $(z_{0},\theta)$ is sub-solution pair for $(P)_{\lambda}$ and the result is proved.
\end{proof}

\begin{remark}
We would like to point that is possible to use the function $\phi$ from the proof of Theorem \ref{teo-sublinear} to consider  problem $(P)_{\lambda}$ but in order to do this  more restrictions on the functions $p$ and $f$  are needed.
\end{remark}


\end{document}